\documentclass[12pt]{amsart}
\usepackage{amssymb}

\newlength\tindent
\usepackage{amsmath}
\usepackage{centernot}
\usepackage{mathtools}

\usepackage{amsfonts}
\usepackage{bbm}
\usepackage{bbold}
\usepackage[utf8]{inputenc}
\usepackage{amsthm}
\usepackage[colorlinks]{hyperref}
\usepackage{pst-node}
\usepackage{tikz-cd} 

\DeclareSymbolFont{bbold}{U}{bbold}{m}{n}
\DeclareSymbolFontAlphabet{\mathbbold}{bbold}

\begin{document}

\theoremstyle{plain}
\newtheorem{theorem}{Theorem}[section]
\newtheorem{proposition}[theorem]{Proposition}
\newtheorem{corollary}[theorem]{Corollary}
\newtheorem{lemma}[theorem]{Lemma}
\newtheorem{conjecture}[theorem]{Conjecture}

\theoremstyle{definition}
\newtheorem{remark}[theorem]{Remark}
\newtheorem{question}[theorem]{Question}
\newtheorem{numbered}[theorem]{}
\newtheorem{example}[theorem]{Example}
\newtheorem{defn}[theorem]{Definition}
\title[Unbounded topologies and uo-convergence]
{Unbounded topologies and uo-convergence in locally solid vector lattices}

\author{Mitchell A. Taylor}
\address{Department of Mathematical and Statistical Sciences,
         University of Alberta, Edmonton, AB, T6G\,2G1, Canada.}
\email{mataylor@ualberta.ca}
\thanks{The author acknowledges support from NSERC and the University of Alberta.}

\keywords{ Banach
  lattice, vector lattice,
  un-topology, unbounded order convergence, locally solid topology.} \subjclass[2010]{Primary: 46A40. Secondary: 46B42}

\date{\today}

\begin{abstract}
Suppose $X$ is a vector lattice and there is a notion of convergence $x_{\alpha} \xrightarrow{\sigma} x$ in $X$. Then we can speak of an ``unbounded" version of this convergence by saying that $x_{\alpha} \xrightarrow{u\sigma} x$ if $\lvert x_\alpha-x\rvert \wedge u\xrightarrow{\sigma} 0$ for every $u \in X_+$. In the literature, the unbounded versions of the norm, order and absolute weak convergence have been studied. Here we create a general theory of unbounded convergence but with a focus on $uo$-convergence and those convergences deriving from locally solid topologies. We will see that, not only do the majority of recent results on unbounded norm convergence generalize, but they do so effortlessly. Not only that, but the stucture of unbounded topologies is clearer without a norm. We demonstrate this by removing metrizability, completeness, and local convexity from nearly all arguments, while at the same time making the proofs simpler and more general. We also give characterizations of minimal topologies in terms of unbounded topologies and $uo$-convergence.
\end{abstract}

\maketitle

\section{Preliminaries}

A net $(x_{\alpha})_{\alpha \in A}$ in a vector lattice $X$ is \textbf{\textit{order convergent}} to $x \in X$ if there exists a net $(y_{\beta})_{\beta \in B}$, possibly over a different index set, such that $y_{\beta} \downarrow 0$ and for each $\beta \in B$ there exists $\alpha_0 \in A$ satisfying $\lvert x_{\alpha}-x\rvert \leq y_{\beta}$ for all $\alpha \geq \alpha_0$. We will write $x_{\alpha} \xrightarrow{o}x$ to denote this convergence. Recall that a net $(x_{\alpha})$ in a vector lattice $X$ is \textbf{\textit{unbounded $o$-convergent}} (or \textbf{\textit{$uo$-convergent}}) to $x\in X$ if $\lvert x_{\alpha}-x\rvert \wedge u \xrightarrow{o} 0$ for all $u \in X_+$. We will write $x_{\alpha} \xrightarrow{uo} x$ to indicate that the net $(x_{\alpha})$ $uo$-converges to $x$. It is clear that order convergent nets are $uo$-convergent. It is also known that these convergences may fail to be topological; there may be no topology on $X$ such that nets converge, say, in order, if and only if they converge with respect to the topology. This fact will distinguish $uo$-convergence from the rest of the unbounded convergences described in this paper.  For general results on $uo$-convergence we refer the reader to [\ref{708}] and [\ref{716}].
\\

In classical literature such as [\ref{701}], [\ref{705}] and [\ref{709}], the definition of order convergence is slightly different. In those books a net $(x_{\alpha})_{\alpha \in A}$ is said to be order convergent to $x\in X$ if there exists a net $(y_{\alpha})_{\alpha \in A}$ such that $y_{\alpha} \downarrow 0$ and $\lvert x_{\alpha}-x\rvert \leq y_{\alpha}$ for all $\alpha \in A$. We will write $x_{\alpha} \xrightarrow{o_1}x$ to distinguish this convergence, but it will not come up often. As needed, it will be shown that properties of locally solid vector lattices are independent of the definition of order convergence. Keeping in mind this slight discrepancy, the reader is referred to [\ref{705}] for all undefined terms. Throughout this paper, all vector lattices are assumed Archimedean.

\section{Basic results on unbounded locally solid topologies}
\begin{defn}
Suppose that $X$ is a vector lattice and $\tau$ is a (not necessarily Hausdorff) linear topology on $X$. We say that a net $(x_{\alpha})\subseteq X$ is \textbf{\textit{unbounded $\tau$-convergent}} to $x\in X$ if $\lvert x_{\alpha}-x\rvert \wedge u\xrightarrow{\tau}0$ for all $u \in X_+$. 
\end{defn}
\begin{defn}
A (not necessarily Hausdorff) topology $\tau$ on a vector lattice $X$ is said to be \textbf{\textit{locally solid}} if it is linear and has a base at zero consisting of solid sets.
\end{defn}
The next theorem justifies the interest in unbounded convergences deriving from locally solid topologies.
\begin{theorem}\label{1}
If $\tau$ is a locally solid topology on a vector lattice $X$ then the unbounded $\tau$-convergence is also a topological convergence on $X$. Moreover, the corresponding topology, $u\tau$, is locally solid. It is Hausdorff if and only if $\tau$ is.
\end{theorem}
\begin{proof}
Since $\tau$ is locally solid it has a base $\{U_i\}_{i\in I}$ at zero consisting of solid neighbourhoods. For each $i \in I$ and $u \in X_+$ define $U_{i,u}:=\{x \in X: \lvert x\rvert \wedge u \in U_i\}$. We claim that the collection $\mathcal{N}_0:=\{U_{i,u} : i\in I, u\in X_+\}$ is a base of neighbourhoods of zero for a locally solid topology; we will call it $u\tau$. Notice that $(x_{\alpha})$ unbounded $\tau$-converges to $0$ iff every set in $\mathcal{N}_0$ contains a tail of the net. After noting that the unbounded $\tau$-convergence is translation invariant, this means the unbounded $\tau$-convergence is exactly the convergence given by this topology. Notice also that $U_i \subseteq U_{i,u}$ and, since $U_i$ is solid, so is $U_{i,u}$.
\\

We now verify that $\mathcal{N}_0$ is a base at zero. Trivially, every set in $\mathcal{N}_0$ contains $0$. We now show that the intersection of any two sets in $\mathcal{N}_0$ contains another set in $\mathcal{N}_0$. Take $U_{i,u_1}, U_{j,u_2} \in \mathcal{N}_0$. Then $U_{i,u_1}\cap U_{j,u_2}=\{x \in X: \lvert x\rvert \wedge u_1\in U_i \ \& \ \lvert x\rvert \wedge u_2 \in U_j\}$. Since $\{U_i\}$ is a base we can find $k\in I$ such that $U_k \subseteq U_i \cap U_j$. We claim that $U_{k,u_1 \vee u_2} \subseteq U_{i,u_1}\cap U_{j,u_2}.$ Indeed, if $x \in U_{k,u_1 \vee u_2}$, then $\lvert x\rvert \wedge(u_1\vee u_2) \in U_k \subseteq U_i \cap U_j$. Therefore, since $\lvert x\rvert \wedge u_1 \leq \lvert x\rvert \wedge(u_1\vee u_2) \in U_i \cap U_j \subseteq U_i$ and $U_i$ is solid, we have $x \in U_{i,u_1}$. Similarly, $x \in U_{j,u_2}$.\\

We know that for every  $i$ there exists $j$ such that  $U_j+U_j \subseteq U_i$. From this we deduce  that for all $i$ and all $u$,  if  $x,y \in U_{j,u}$ then
\begin{equation}
\lvert x+y\rvert \wedge u \leq \lvert x\rvert \wedge u+\lvert y\rvert  \wedge u \in U_j+U_j \subseteq U_i
\end{equation}
so that $U_{j,u}+U_{j,u} \subseteq U_{i,u}$.
\\

If $\lvert \lambda\rvert\leq 1$ then $\lambda U_{i,u} \subseteq U_{i,u}$ because $U_{i,u}$ is solid. It follows from $U_i \subseteq U_{i,u}$ that $U_{i,u}$ is absorbing. This completes the verification by [\ref{701}] Theorem~5.6.
\\

Suppose further that $\tau$ is Hausdorff; we will verify that $\bigcap \mathcal{N}_0=\{0\}.$ Indeed, suppose that $x \in U_{i,u}$ for all $i \in I$ and $u \in X_+$. In particular, $x \in U_{i,\lvert x\rvert}$ which means that $\lvert x\rvert \in U_i$ for all $i \in I$. Since $\tau$ is Hausdorff, $\bigcap U_i=\{0\}$ and we conclude that $x=0$.
\\

Finally, if $u\tau$ is Hausdorff then $\tau$ is Hausdorff since $U_i \subseteq U_{i,u}$.
\end{proof}
\begin{remark}
If $\tau$ is the norm topology on a Banach lattice $X$, the corresponding $u\tau$-topology is called $un$-topology; it has been studied in [\ref{704}], [\ref{710}] and [\ref{706}]. It is easy to see that the weak and absolute weak topologies on $X$ generate the same unbounded convergence and, since the absolute weak topology is locally solid, this convergence is topological. It has been denoted $uaw$ and was studied in [\ref{702}].
\end{remark}
From now on, unless explicitly stated otherwise, throughout this paper the minimum assumption is that $X$ is an Archimedean vector lattice and $\tau$ is locally solid. The following straightforward result should be noted. It justifies the name \emph{unbounded} $\tau$-convergence.
\begin{proposition}\label{8}
If $x_{\alpha}\xrightarrow{\tau}0$ then $x_{\alpha}\xrightarrow{u\tau}0$. For order bounded nets the convergences agree.
\end{proposition}

\begin{remark} Observe that $uu\tau=u\tau$, so there are no chains of unbounded topologies. To see this note that $x_{\alpha}\xrightarrow{uu\tau}x$ means that for any $u \in X_+$, $\lvert x_{\alpha}-x\rvert \wedge u\xrightarrow{u\tau}0$. Since the net $(\lvert x_{\alpha}-x\rvert \wedge u)$ is order bounded, this is the same as $\lvert x_{\alpha}-x\rvert \wedge u\xrightarrow{\tau}0$, which means $x_{\alpha}\xrightarrow{u\tau}x$. In another language, the map $\tau \mapsto u\tau$ from the set of locally solid topologies on $X$ to itself is idempotent. We give a name to the fixed points or, equivalently, the range, of this map:
\end{remark}
\begin{defn}\label{19}
A locally solid topology $\tau$ is \textbf{\textit{unbounded}} if $\tau=u\tau$ or, equivalently, if $\tau=u\sigma$ for some locally solid topology $\sigma$.
\end{defn}
We next present a few easy corollaries of Theorem~\ref{1} for use later in the paper. 
\begin{corollary}\label{1.1}
Lattice operations are uniformly continuous with respect to $u\tau$, and $u\tau$-closures of solid sets are solid.
\end{corollary}
\begin{proof}
The result follows immediately from Theorem 8.41 and Lemma 8.42 in [\ref{701}].
\end{proof}
In the next corollary, the Archimedean property is not assumed. Statement (iii) states that, under very mild topological assumptions, it is satisfied automatically. Statements (ii) and (iv) are efficient generalizations of Lemma 1.2 and Proposition 4.8 in [\ref{706}].
\begin{corollary}\label{1.2}
Suppose $\tau$ is both locally solid and Hausdorff then:
\begin{enumerate}
\item The positive cone $X_+$ is $u\tau$-closed;
\item If $x_{\alpha} \uparrow$ and $x_{\alpha}\xrightarrow{u\tau}x$, then $x_{\alpha}\uparrow x$;
\item $X$ is Archimedean;
\item Every band in $X$ is $u\tau$-closed.
\end{enumerate}
\end{corollary}
\begin{proof}
The result follows immediately from Theorem 8.43 in [\ref{701}].
\end{proof}

We next work towards a version of Proposition 3.15 in [\ref{708}] that is applicable to locally solid topologies. The proposition is recalled here along with a definition.
\begin{proposition}\label{2}
Let $X$ be a vector lattice, and $Y$ a sublattice of $X$. Then $Y$ is $uo$-closed in $X$ if and only if it is $o$-closed in $X$.
\end{proposition}

\begin{defn}\label{3}
A locally solid topology $\tau$ on a vector lattice is said to be \textbf{\textit{Lebesgue}} (or \textbf{\textit{order continuous}}) if $x_{\alpha} \xrightarrow{o} 0$ implies $x_{\alpha} \xrightarrow{\tau} 0$.
\end{defn}
Note that the Lebesgue property is independent of the definition of order convergence because, as is easily seen, no matter which definition of order convergence is used, it is equivalent to the property that $x_{\alpha}\xrightarrow{\tau}0$ whenever $x_{\alpha}\downarrow 0$.
\begin{proposition}\label{4}
Let $X$ be a vector lattice, $\tau$ a Hausdorff locally solid topology on $X$, and $Y$ a sublattice of $X$. $Y$ is (sequentially) $u\tau$-closed in $X$ if and only if it is (sequentially) $\tau$-closed in $X$.
\end{proposition}
\begin{proof}
If $Y$ is $u\tau$-closed in $X$ it is clearly $\tau$-closed in $X$. Suppose now that $Y$ is $\tau$-closed in $X$ and let $(y_{\alpha})$ be a net in $Y$ that $u\tau$-converges in $X$ to some $x \in X$. Since lattice operations are $u\tau$-continuous we have that $y_{\alpha}^{\pm}\xrightarrow{u\tau}x^{\pm}$ in $X$. Thus, WLOG, we may assume that $(y_{\alpha})\subseteq Y_+$ and $x \in X_+$. Observe that for every $z\in X_+$,
\begin{equation}\label{7678}
\lvert y_{\alpha}\wedge z-x \wedge z\rvert \leq \lvert y_{\alpha}-x\rvert \wedge z \xrightarrow{\tau} 0.
\end{equation}
In particular, for any $y \in Y_+$, $y_{\alpha}\wedge y \xrightarrow{\tau} x\wedge y$. Since $Y$ is $\tau$-closed, $x \wedge y\in Y$ for any $y \in Y_+$. 

On the other hand, taking $z=x$ in (\ref{7678}) we get that $y_{\alpha}\wedge x\xrightarrow{\tau} x$. Since we have just shown that $y_{\alpha}\wedge x\in Y$, it follows that $x\in \overline{Y}^\tau=Y$. The same proof works for sequences and for $o/uo$-convergence.

\end{proof}

Proposition~\ref{4} will be used to prove a much deeper statement: see Theorem~\ref{26}. 


To establish the ease with which results transfer from $un$ to $u\tau$ we next proceed to generalize many results from the aforementioned papers on $un$-convergence. Later in the paper we will justify the need to look at unbounded topologies other than $un$, and how working only with Banach lattices can be too restrictive.
\begin{defn}\label{5}
A locally solid topology $\tau$ on a vector lattice is said to be \textbf{\textit{$uo$-Lebesgue}} (or \textbf{\textit{unbounded order continuous}}) if $x_{\alpha} \xrightarrow{uo} 0$ implies $x_{\alpha} \xrightarrow{\tau} 0$.
\end{defn}

It is clear that the $uo$-Lebesgue property implies the Lebesgue property but not conversely:

\begin{example}\label{6}
The norm topology of $c_0$ is order continuous but not unbounded order continuous.
\end{example}
\begin{proposition}\label{7}
If $\tau$ is Lebesgue then $u\tau$ is $uo$-Lebesgue. In particular, $u\tau$ is Lebesgue.
\end{proposition}
\begin{proof}
Suppose $x_{\alpha} \xrightarrow{uo} x$, i.e., $\forall u \in X_+$, $\lvert x_{\alpha}-x\rvert \wedge u \xrightarrow{o} 0$. The Lebesgue property implies that $\lvert x_{\alpha}-x\rvert \wedge u \xrightarrow{\tau} 0$ so that $x_{\alpha} \xrightarrow{u\tau}x$. 
\end{proof}

There is much more to say about the $uo$-Lebesgue property and, in fact, a whole section on it. We continue now with more easy observations.
\\

Next we present Lemmas 2.1 and 2.2 of [\ref{706}] which carry over with minor modification. The proofs are similar and, therefore, omitted.
\begin{lemma}\label{9}
Let $X$ be a vector lattice, $u \in X_+$ and $U$ a solid subset of $X$. Then $U_u:=\{x\in X:\lvert x \rvert \wedge u \in U\}$ is either contained in $[-u,u]$ or contains a non-trivial ideal. If $U$ is, further, absorbing, and $U_u$ is contained in $[-u,u]$, then $u$ is a strong unit.
\end{lemma}

Next we present a trivialized version of Thereom 2.3 in [\ref{706}].
\begin{proposition}\label{11}
Let $(X,\tau)$ be a locally solid vector lattice and suppose that $\tau$ has a neighbourhood $U$ of zero containing no non-trivial ideal. If there is a $u\tau$-neighbourhood contained in $U$ then $X$ has a strong unit.
\end{proposition}
\begin{proof}
Let $\{U_i\}$ be a solid base at zero for $\tau$ and suppose there exists $i$ and $u>0$ s.t. $U_{i,u} \subseteq U$. We conclude that $U_{i,u}$ contains no non-trivial ideal and, therefore, $u$ is a strong unit.
\end{proof}
This allows us to prove that $u\tau$-neighbourhoods are generally quite large.
\begin{corollary}
If $\tau$ is unbounded and $X$ does not admit a strong unit then every neighbourhood of zero for $\tau$ contains a non-trivial ideal.
\end{corollary}
\begin{defn}
A subset $A$ of a locally solid vector lattice $(X,\tau)$ is \textbf{\textit{$\tau$-almost order bounded}} if for every solid $\tau$-neighbourhood $U$ of zero there exists $u \in X_+$ with $A \subseteq [-u,u]+U$.
\end{defn}
It is easily seen that for solid $U$, $x \in [-u,u]+U$ is equivalent to $\left(\lvert x\rvert-u\right)^+\in U$. The proof is the same as the norm case. This leads to a generalization of Lemma 2.9 in [\ref{704}]; the proof is left to the reader.
\begin{proposition}\label{123456789}
If $x_{\alpha}\xrightarrow{u\tau}x$ and $(x_{\alpha})$ is $\tau$-almost order bounded then $x_{\alpha}\xrightarrow{\tau}x$.
\end{proposition}

In a similar vein, the following can easily be proved; just follow the proof of Proposition 3.7 in [\ref{716}] or notice it is an immediate corollary of Proposition~\ref{123456789}.
\begin{proposition}\label{45623}
Let $(X,\tau)$ be a locally solid vector lattice with the Lebesgue property. If $(x_{\alpha})$ is $\tau$-almost order bounded and $uo$-converges to $x$, then $(x_{\alpha})$ $\tau$-converges to $x$.
\end{proposition}
One direction of [\ref{704}] Theorem 4.4 can also be generalized. The proof is, again, easy and left to the reader.
\begin{proposition}
Let $(x_n)$ be a sequence in $(X,\tau)$ and assume $\tau$ is Lebesgue. If every subsequence of $(x_n)$ has a further subsequence which is $uo$-null then $(x_n)$ is $u\tau$-null.
\end{proposition}

Recall that a net $(x_{\alpha})$ in a vector lattice $X$ is \textbf{\textit{$uo$-Cauchy}} if the net $(x_{\alpha}-x_{\alpha'})_{(\alpha,\alpha')}$ $uo$-converges to zero. $X$ is \textbf{\textit{$uo$-complete}} if every $uo$-Cauchy net is $uo$-convergent. A study of $uo$-complete spaces was undertaken in [\ref{712}]. A weaker property involving norm boundedness was introduced in [\ref{708}]. Here is a generalization of both definitions to locally solid vector lattices. 
\begin{defn}
A locally solid vector lattice $(X,\tau)$ is \textbf{\textit{boundedly $uo$-complete}} (respectively, \textbf{\textit{sequentially boundedly $uo$-complete}}) if every $\tau$-bounded $uo$-Cauchy net (respectively, sequence) is $uo$-convergent.
\end{defn}
\begin{proposition}
Let $(X,\tau)$ be a locally solid vector lattice. If $(X,\tau)$ is boundedly $uo$-complete then it is order complete. If $(X,\tau)$ is sequentially boundedly $uo$-complete then it is $\sigma$-order complete.
\end{proposition}
\begin{proof}
Let $(x_{\alpha})$ be a net in $X$ such that $0\leq x_{\alpha}\uparrow\leq x$ for some $x\in X$. By [\ref{705}] Theorem 2.19, $(x_{\alpha})$ is $\tau$-bounded. By [\ref{712}] Lemma 2.1, $(x_{\alpha})$ is order Cauchy and hence $uo$-Cauchy. By the assumption that $(X,\tau)$ is boundedly $uo$-complete and the order boundedness of $(x_{\alpha})$, we conclude that $x_{\alpha}\xrightarrow{o}y$ for some $y\in X$. Since $(x_{\alpha})$ is increasing, $y=\sup x_{\alpha}$. The sequential argument is similar.
\end{proof}
Notice that a vector lattice $X$ is $uo$-complete if and only if $X$ equipped with the trivial topology (which is locally solid) is boundedly $uo$-complete. Thus, this is a more general concept than both $uo$-complete vector lattices and boundedly $uo$-complete Banach lattices. Notice also that the order completeness assumption in [\ref{712}] Proposition 2.8 may now be dropped:
\begin{corollary}
Let $X$ be a vector lattice. If $X$ is $uo$-complete then it is universally complete. Conversely, if $X$ is universally complete, and, in addition, has the countable sup property, then it is $uo$-complete.
\end{corollary}

It is easy to see that a net is order null iff it is $uo$-null and has an order bounded tail. This is why it is of interest to consider topologically bounded $uo$-null nets, as topological boundedness acts as an approximation to order boundedness. Recall that a  locally solid vector lattice $(X,\tau)$ is said to satisfy the \textbf{\textit{Levi property}}\footnote{This property generalizes the concept of monotonically complete Banach lattices appearing in [\ref{MN}].} if every increasing $\tau$-bounded net of $X_+$ has a supremum in $X$. The Levi and Fatou properties together are enough to ensure that a space is boundedly $uo$-complete. The formal statement is Theorem~\ref{86868686}. Recall that a locally solid topology is \textbf{\textit{Fatou}} if it has a base at zero consisting of solid order closed sets. Although it is not obvious by definition, the Fatou property is independent of the definition of order convergence. This can be easily deduced as an argument similar to that of Lemma 1.15 in [\ref{705}] shows that a solid set is $o$-closed if and only if it is $o_1$-closed. In fact, more is true. By reviewing the arguments in section 3 of the Bachelor's thesis [\ref{Imhoff}], one can prove that a set $A$ in a vector lattice $X$ is $o_1$-closed if and only if it is $o$-closed.
\\

\begin{remark}\label{13376} It should be noted that if $\tau$ is Hausdorff then every $\tau$-convergent $uo$-Cauchy net $uo$-converges to its $\tau$-limit. This follows since lattice operations are $\tau$-continuous and the positive cone is $\tau$-closed: see [\ref{705}] Theorem 2.21. The following is a slight generalization of Proposition 4.2 in [\ref{716}]. The proof is similar but is provided for convienence of the reader.\end{remark}
\begin{proposition}
Suppose that $\tau$ is a complete Hausdorff Lebesgue topology on a vector lattice $X$. If $(x_{\alpha})$ is a $\tau$-almost order bounded $uo$-Cauchy net in $X$ then $(x_{\alpha})$ converges $uo$ and $\tau$ to the same limit.
\end{proposition}
\begin{proof}
Suppose $(x_{\alpha})$ is $\tau$-almost order bounded and $uo$-Cauchy. Then the net $(x_{\alpha}-x_{\alpha'})$ is $\tau$-almost order bounded and is $uo$-convergent to zero. By Proposition~\ref{45623}, $(x_{\alpha}-x_{\alpha'})$ is $\tau$-null. It follows that $(x_{\alpha})$ is $\tau$-Cauchy and thus $\tau$-convergent to some $x\in X$ since $\tau$ is complete. By Remark~\ref{13376}, $(x_{\alpha})$ $uo$-converges to $x$.
\end{proof}

\section{Products and sublattices}
\subsection{Products}
Let $\{(X_{\alpha}, \tau_{\alpha})\}_{\alpha\in A}$ be a family of locally solid vector lattices and let $X=\prod X_{\alpha}$ be the Cartesian product, ordered componentwise, and equipped with the product topology $\prod \tau_{\alpha}$. It is known that $X$ has the structure of a locally solid vector lattice. See [\ref{705}] pages 8 and 56 for details.
\begin{theorem}\label{15}
Let $\{(X_{\alpha},\tau_{\alpha})\}$ be a family of locally solid vector lattices. Then $(\prod X_{\alpha}, u\prod \tau_{\alpha})=(\prod X_{\alpha}, \prod u\tau_{\alpha})$.
\end{theorem}
\begin{proof}
Let $\{U_i^{\alpha}\}_{i \in I_{\alpha}}$ be a solid base for $(X_{\alpha},\tau_{\alpha})$ at zero. We know the following:
\\

$\{U_{i,u}^{\alpha}\}_{i\in I_{\alpha}, u\in {X_{\alpha}}_+}$ is a solid base for $(X_{\alpha}, u\tau_{\alpha})$ at zero where $U_{i,u}^{\alpha}=\{x \in X_{\alpha}: \lvert x\rvert \wedge u \in U_i^{\alpha}\}$.
\\

A solid base of $(\prod X_{\alpha}, \prod u\tau_{\alpha})$ at zero consists of sets of the form $\prod U^{\alpha}$  where $U^{\alpha}=X_{\alpha}$ for all but finitely many $\alpha$ and if $U^{\alpha} \neq X_{\alpha}$ for some $\alpha$ then $U^{\alpha}=U_{i,u}^{\alpha}$ for some $i \in I_{\alpha}$ and $u \in {X_{\alpha}}_+$.
\\

A solid base for $(\prod X_{\alpha},\prod \tau_{\alpha})$ at zero consists of sets of the form $\prod V^{\alpha}$ where $V^{\alpha}=X_{\alpha}$ for all but finitely many $\alpha$ and if $V^{\alpha} \neq X_{\alpha}$ for some $\alpha$ then $V^{\alpha}=U^{\alpha}_i$ for some $i \in I_{\alpha}$. Therefore, a solid base for $(\prod X_{\alpha},u\prod \tau_{\alpha})$ at zero consists of sets of the form $(\prod V^{\alpha})_w$ where $w=(w_{\alpha})\in (\prod X_{\alpha})_+=\prod  {X_{\alpha}}_+$ and $(\prod V^{\alpha})_w=\{x=(x_{\alpha}) \in \prod X_{\alpha}: \lvert (x_{\alpha})\rvert \wedge (w_{\alpha})=(\lvert x_{\alpha}\rvert \wedge w_{\alpha}) \in \prod V^{\alpha}\}.$ Here we used the fact that lattice operations are componentwise.
\\

The theorem follows easily from this. Consider a set $\prod U^{\alpha}$ and assume $U^{\alpha}=X_{\alpha}$ except for the indices $\alpha_1, \dots, \alpha_n$ where $U^{\alpha_j}=U^{\alpha_j}_{i_j,u_j}$ for $j\in \{1,\dots, n\}$, $i_j\in I_{\alpha_j}$ and $u_j \in {X_{\alpha_j}}_+$. Then $\prod U^{\alpha}=(\prod V^{\alpha})_w$ where $w_{\alpha_j}=u_j$ for $j=1,\dots, n$ and $w_{\alpha}=0$ otherwise and $V^{\alpha_j}=U^{\alpha_j}_{i_j}$ for $j=1\dots, n$ and $V^{\alpha}=X_{\alpha}$ otherwise.
\\

Conversely, consider a set of the form $(\prod V^{\alpha})_w$ and assume $V^{\alpha}=X_{\alpha}$ except for the indices $\alpha_1,\dots, \alpha_n$ in which case $V^{\alpha_j}=U_{i_j}^{\alpha_j}$ for $j \in\{1,\dots,n\}$ and $i_j \in I_{\alpha_j}$. Then $(\prod V^{\alpha})_w=\prod U^{\alpha}$ where $U^{\alpha_j}=U^{\alpha_j}_{i_j, w_{\alpha_j}}$ for $j \in \{1,\dots, n\}$ and $U^{\alpha}=X_{\alpha}$ otherwise.
\\

Since each base is contained in the other, the topologies agree.
\end{proof}

\begin{remark}
Although unbounded topologies behave very well in the product, less is known about topological completions and quotients. It will soon be proved that the completion of a Hausdorff unbounded Lebesgue topology is a Hausdorff unbounded Lebesgue topology, but in general the picture is unclear.
\end{remark}
\begin{question}
Let $(X,\tau)$ be a Hausdorff locally solid vector lattice, and assume $\tau$ is unbounded. Under what conditions on $\tau$ (Fatou, pre-Lebesgue, $\sigma$-Lebesgue, none) is the topological completion of $(X,\tau)$ unbounded? Give conditions on $\tau$ for the quotient of $(X,\tau)$ by a closed ideal of $X$ to be unbounded. Is this true if $\tau$ is Lebesgue? 
\end{question}

\subsection{Sublattices}
Let $Y$ be a sublattice of a locally solid vector lattice $(X,\tau)$. The reader should convince themselves that $Y$, equipped with the subspace topology, $\tau|_Y$, is a locally solid vector lattice in its own right. It would be natural to now compare $u(\tau|_Y)$ and $(u\tau)|_Y$, but this was already implicitly done in [\ref{706}]. In general, $u(\tau|_Y) \subsetneq (u\tau)|_Y$, even if $Y$ is a band. If $(y_{\alpha})$ is a net in $Y$ we will write $y_{\alpha}\xrightarrow{u\tau}0$ in $Y$ to mean $y_{\alpha} \rightarrow0$ in $(Y,u(\tau|_Y)).$ We now look for conditions that make all convergences agree.
\begin{lemma}\label{18}
Let $Y$ be a sublattice of a locally solid vector lattice $(X,\tau)$ and $(y_{\alpha})$ a net in $Y$ such that $y_{\alpha} \xrightarrow{u\tau} 0$ in $Y$. Each of the following conditions implies that $y_{\alpha} \xrightarrow{u\tau} 0$ in $X$.
\begin{enumerate}
\item $Y$ is majorizing in $X$;
\item $Y$ is $\tau$-dense in $X$;
\item $Y$ is a projection band in $X$.

\end{enumerate}
\end{lemma}
\begin{proof}
WLOG, $y_{\alpha} \geq 0$ for every $\alpha$. (i) gives no trouble. To prove (ii), take $u \in X_+$ and fix solid $\tau$-neighbourhoods $U$ and $V$ of zero (in $X$) with $V+V \subseteq U$. Since $Y$ is dense in $X$ we can find a $v \in Y$ with $v-u \in V$. WLOG, $v \in Y_+$ since $V$ is solid and $\lvert\lvert v\rvert-u\rvert=\lvert\lvert v\rvert-\lvert u\rvert\rvert \leq \lvert v-u\rvert \in V$. By assumption, $y_{\alpha} \wedge v \xrightarrow{\tau}0$ so we can find $\alpha_0$ such that $y_{\alpha} \wedge v \in V$ whenever $\alpha\geq \alpha_0$. It follows from $u \leq v+\lvert u-v\rvert$ that $y_{\alpha}\wedge u\leq y_{\alpha}\wedge v+\lvert u-v\rvert.$ This implies that $y_{\alpha} \wedge u \in U$ for all $\alpha \geq \alpha_0$ since
\begin{equation}
0 \leq y_{\alpha} \wedge u \leq y_{\alpha}\wedge v +\lvert u-v\rvert \in V+V \subseteq U
\end{equation}
where, again, we used that $U$ and $V$ are solid. This means that $y_{\alpha}\wedge u\xrightarrow{\tau}0$. Hence, $y_{\alpha}\xrightarrow{u\tau}0$ in $X$.
\\

To prove (iii), let $u \in X_+$. Then $u=v+w$ for some positive $v \in Y$ and $w \in Y^d$. It follows from $y_\alpha\perp w$ that $y_{\alpha} \wedge u=y_{\alpha} \wedge v \xrightarrow{\tau}0.$
\end{proof}
Let $X$ be a vector lattice, $\tau$ a locally solid topology on $X$, and $X^{\delta}$ the order completion of $X$. It is known that one can find a locally solid topology, say, $\tau^*$, on $X^{\delta}$ that extends $\tau$. See Exercise $8$ on page $73$ of [\ref{705}] for details on how to construct such an extension. Since $X$ is majorizing in $X^{\delta}$, Lemma~\ref{18} gives the following.
\begin{corollary}\label{18.1}
If $(X,\tau)$ is a locally solid vector lattice and $(x_{\alpha})$ is a net in $X$ then $x_{\alpha} \xrightarrow{u\tau}0$ in $X$ if and only if $x_{\alpha} \xrightarrow{u(\tau^*)}0$ in $X^{\delta}$. Here $\tau^*$ denotes a locally solid extension of $\tau$ to $X^{\delta}$. 
\end{corollary}

\section{Pre-Lebesgue property and disjoint sequences}
Recall the following definition from page 75 of [\ref{705}]:
\begin{defn}\label{22}
Let $(X,\tau)$ be a locally solid vector lattice. We say that $(X,\tau)$ satisfies the \textbf{\textit{pre-Lebesgue property}} (or that $\tau$ is a \textbf{\textit{pre-Lebesgue topology)}}, if $0\leq x_n\! \uparrow \leq x$ in $X$ implies that $(x_n)$ is a $\tau$-Cauchy sequence.
\end{defn}
Recall that Theorem 3.23 of [\ref{705}] states that in an Archimedean locally solid vector lattice the Lebesgue property implies the pre-Lebesgue property. It is also known that in a topologically complete Hausdorff locally solid vector lattice that the Lebesgue property is equivalent to the pre-Lebesgue property and these spaces are always order complete. This is Theorem 3.24 of [\ref{705}]. The next theorem tells us exactly when disjoint sequences are $u\tau$-null. Parts (i)-(iv) are Theorem 3.22 of [\ref{705}], (v) and (vi) are new.
\begin{theorem}\label{23}
For a locally solid vector lattice $(X, \tau)$ TFAE:
\begin{enumerate}
\item $(X,\tau)$ satisfies the pre-Lebesgue property;

\item If $0 \leq x_{\alpha}\uparrow \leq x$ holds in $X$, then $(x_{\alpha})$ is a $\tau$-Cauchy net of $X$;
\item Every order bounded disjoint sequence of $X$ is $\tau$-convergent to zero;
\item Every order bounded $k$-disjoint sequence of $X$ is $\tau$-convergent to zero;
\item Every disjoint sequence in $X$ is $u\tau$-convergent to zero;
\item Every disjoint net in $X$ is $u\tau$-convergent to zero.\footnote{In statements such as this we require that the index set of the net has no maximal elements. See page 9 of [\ref{705}] for a further discussion on this minor issue.}
\end{enumerate}
\end{theorem}
\begin{proof}
(iii)$\Rightarrow$(v): Suppose $(x_n)$ is a disjoint sequence. For every $u\in X_+$, $(\lvert x_n\rvert \wedge u)$ is order bounded and disjoint, so is $\tau$-convergent to zero. This proves $x_n \xrightarrow{u\tau}0$.
\\

(v) $\Rightarrow$ (iii): Let $(x_n)\subseteq [-u,u]$ be a disjoint order bounded sequence. By (v), $x_n \xrightarrow{u\tau}0$ so, in particular, $\lvert x_n\rvert=\lvert x_n\rvert \wedge u \xrightarrow{\tau}0$. This proves $(x_n)$ is $\tau$-null.
\\

Next we prove (v)$\Leftrightarrow$(vi). Clearly (vi)$\Rightarrow$(v). Assume (v) holds and suppose there exists a disjoint net $(x_{\alpha})$ which is not $u\tau$-null. Let $\{U_i\}$ be a solid base of neighbourhoods of zero for $\tau$ and $\{U_{i,u}\}$ the solid base for $u\tau$ described in Theorem~\ref{1}. Since $(x_{\alpha})$ is not $u\tau$-null there exists $U_{i,u}$ such that for every $\alpha$ there exists $\beta > \alpha$ with $x_{\beta} \notin U_{i,u}$. Inductively, we find an increasing sequence $(\alpha_k)$ of indices such that $x_{\alpha_k} \notin U_{i,u}$. Hence the sequence $(x_{\alpha_k})$ is disjoint but not $u\tau$-null.
\end{proof}
Since, for a Banach lattice, the norm topology is complete, the pre-Lebesgue property agrees with the Lebesgue property. This theorem can therefore be thought of as a generalization of Proposition 3.5 in [\ref{706}]. Theorem~\ref{23} has the following corollaries:
\begin{corollary}\label{10201}
$\tau$ has the pre-Lebesgue property if and only if $u\tau$ does.
\end{corollary}
\begin{proof}
$\tau$ and $u\tau$-convergences agree on order bounded sequences. Apply (iii).
\end{proof}
\begin{corollary}\label{1000}
If $\tau$ is pre-Lebesgue and unbounded then every disjoint sequence of $X$ is $\tau$-convergent to 0.
\end{corollary}
\begin{question}
Suppose $(X,\tau)$ is a Hausdorff locally solid vector lattice. If every disjoint sequence of $X$ is $\tau$-null, do $\tau$ and $u\tau$ agree (at least on sequences)?
\end{question}


\subsection{$\sigma$-Lebesgue topologies}
Recall that a locally solid topology $\tau$ is \textbf{\textit{$\sigma$-Lebesgue}} if $x_n\downarrow 0\Rightarrow x_n\xrightarrow{\tau}0$ or, equivalently, $x_n\xrightarrow{o_1}0 \Rightarrow x_n\xrightarrow{\tau}0$. Example 3.25 in [\ref{705}] shows that the $\sigma$-Lebesgue property does not imply the pre-Lebesgue property. It should be noted that an equivalent definition is not obtained if we replace $o_1$-convergence with $o$-convergence in the latter definition of the $\sigma$-Lebesgue property. In other words, the $\sigma$-Lebesgue property is not independent of the definition of order convergence.
\\

By Theorem~\ref{23}(v) it may be tempting to conclude that if $\tau$ is $\sigma$-Lebesgue then $\tau$ is pre-Lebesgue (since disjoint sequences are $uo$-null). This is not the case, however, as the example above illustrates. In the next example we show how the definition of order convergence can effect properties of $uo$-convergence. We say a net $(x_{\alpha})$ in a vector lattice $X$ is $uo_1$-convergent to $x\in X$ if $\lvert x_{\alpha}-x\rvert\wedge u\xrightarrow{o_1}0$ for all $u\in X_+$.
\begin{example}
By $[\ref{708}]$ Corollary $3.6$ every disjoint sequence in a vector lattice $X$ is $uo$-null. We will show that it is not the case that every disjoint sequence in $X$ is $uo_1$-null. Indeed, let $(X,\tau)$ be as in Example 3.25 of [\ref{705}] and assume every disjoint sequence is $uo_1$-null. Then, since $\tau$ is $\sigma$-Lebesgue, every disjoint sequence is $u\tau$-null. By Theorem~\ref{23}(v) $\tau$ has the pre-Lebesgue property, a contradiction.
\end{example}
\begin{proposition}\label{23.1}
Suppose $\tau$ is a locally solid topology. $\tau$ is Lebesgue iff $u\tau$ is. $\tau$ is $\sigma$-Lebesgue iff $u\tau$ is.
\end{proposition}
\begin{proof}
If $x_{\alpha}\downarrow 0$ then, passing to a tail, $(x_{\alpha})$ is order bounded. Therefore, $x_{\alpha}\xrightarrow{\tau}0\Leftrightarrow x_{\alpha}\xrightarrow{u\tau}0$. The sequential proof is similar.
\end{proof}
\section{The $uo$-Lebesgue property and universal completions}
Throughout this section, as usual, $X$ is a vector lattice and all topologies are assumed locally solid. We first deal with the deep connection between universal completions, unbounded topologies and $uo$-convergence. Recall Theorem 7.54 in [\ref{705}]:

\begin{theorem}\label{999} For a vector lattice $X$ we have the following:
\begin{enumerate}
\item $X$ can admit at most one Hausdorff Lebesgue topology that extends to its universal completion as a locally solid topology;
\item $X$ admits a Hausdorff Lebesgue topology if and only if $X^u$ does.
\end{enumerate}
\end{theorem}
We can now add an eighth and nineth equivalence to Theorem $7.51$ in [\ref{705}]. For convenience of the reader, and since we will need nearly all these properties, we recall the entire theorem.  We remark that dominable sets will not play a role in this paper, and a locally solid topology is \textbf{\textit{$\sigma$-Fatou}} if it has a base $\{U_i\}$ at zero consisting of solid sets with the property that $(x_n)\subseteq U_i$ and $0\leq x_n\uparrow x$ implies $x\in U_i$.
\begin{theorem}\label{24.1} For a Hausdorff locally solid vector lattice $(X,\tau)$ with the Lebesgue property the following statements are equivalent.
\begin{enumerate}
\item $\tau$ extends to a Lebesgue topology on $X^u$;
\item $\tau$ extends to a locally solid topology on $X^u$;
\item $\tau$ is coarser than any Hausdorff $\sigma$-Fatou topology on $X$;
\item Every dominable subset of $X_+$ is $\tau$-bounded;
\item Every disjoint sequence of $X_+$ is $\tau$-convergent to zero;
\item Every disjoint sequence of $X_+$ is $\tau$-bounded;
\item The topological completion $\widehat{X}$ of $(X,\tau)$ is Riesz isomorphic to $X^u$, that is, $\widehat{X}$ is the universal completion of $X$;
\item Every disjoint net of $X_+$ is $\tau$-convergent to zero;
\item $\tau$ is unbounded.
\end{enumerate}
\end{theorem}

\begin{proof}
The proof that (v)$\Leftrightarrow $(viii) is the same technique as in Theorem~\ref{23}.
\\

We now prove that (v)$\Leftrightarrow$(ix). Assume $\tau$ is unbounded and is a Hausdorff Lebesgue topology. Since $\tau$ is Hausdorff, $X$ is Archimedean. Since $X$ is Archimedean and $\tau$ is Lebesgue, $\tau$ is pre-Lebesgue. Now apply Corollary~\ref{1000}.
\\

Now assume (v) holds so that every disjoint sequence of $X_+$ is $\tau$-convergent to zero. Since $\tau$ is Hausdorff and Lebesgue, so is $u\tau$. Since $\tau$-convergence implies $u\tau$-convergence, every disjoint positive sequence is $u\tau$-convergent to zero so that, by (ii), $u\tau$ extends to a locally solid topology on $X^u$. We conclude that $\tau$ and $u\tau$ are both Hausdorff Lebesgue topologies that extend to $X^u$ as locally solid topologies. By Theorem~\ref{999}, $\tau=u\tau$.
\end{proof}
Theorem~\ref{24.1}(vii) yields the following:
\begin{corollary}
Let $\tau$ be an unbounded Hausdorff Lebesgue topology on a vector lattice $X$. $\tau$ is complete if and only if $X$ is universally complete.
\end{corollary}
\begin{remark} Compare this with [\ref{706}] Proposition 6.2 and [\ref{705}] Theorem 7.47. It can also be deduced that if $\tau$ is a topologically complete unbounded Hausdorff Lebesgue topology then it is the only Hausdorff Fatou topology on $X$. See Theorem 7.53 of [\ref{705}].
\end{remark}

By Exercise 5 on page 72 of [\ref{705}], if an unbounded Hausdorff Lebesgue topology $\tau$ is extended to a Lebesgue topology $\tau^u$ on $X^u$, then $\tau^u$ is also Hausdorff. By Theorem 7.53 of [\ref{705}] it is the only Hausdorff Lebesgue (even Fatou) topology $X^u$ can admit. It must therefore be unbounded. By the uniqueness of Hausdorff Lebesgue topologies on $X^u$ we deduce uniqueness of unbounded Hausdorff Lebesgue topologies on $X$ since these types of topologies always extend to $X^u$.
\\

We summarize in a theorem:
\begin{theorem}\label{24}
Let $X$ be a vector lattice. We have the following:
\begin{enumerate}
\item $X$ admits at most one unbounded Hausdorff Lebesgue topology. (It admits an unbounded Hausdorff Lebesgue topology if and only if it admits a Hausdorff Lebesgue topology);
\item Let $\tau$ be a Hausdorff Lebesgue topology on $X$. $\tau$ is unbounded if and only if $\tau$ extends to a locally solid topology on $X^u$. In this case, the extension of $\tau$ to $X^u$ can be chosen to be Hausdorff, Lebesgue and unbounded.
\end{enumerate}
\end{theorem} 

\begin{example}
Let $X$ be an order continuous Banach lattice. Both the norm and $un$ topologies are Hausdorff and Lebesgue. Since these topologies generally differ, it is clear that a space can admit more than one Hausdorff Lebesgue topology. Notice, however, that when $X$ is order continuous, $un$ is the same as $uaw$. The reason for this is that $un$ and $uaw$ are two unbounded Hausdorff Lebesgue topologies, so, by the theory just presented, they must coincide.
\end{example}

Recall that every Lebesgue topology is Fatou; this is Lemma 4.2 of [\ref{705}]. Also, if $\tau$ is a Hausdorff Fatou topology on a universally complete vector lattice $X$ then $(X,\tau)$ is $\tau$-complete. This is Theorem 7.50 of [\ref{705}] and can also be deduced from previous facts presented in this paper.
\begin{corollary}\label{24.2}
Suppose $(X,\tau)$ is Hausdorff and Lebesgue. Then $u\tau$ extends to an unbounded Hausdorff Lebesgue topology $(u\tau)^u$ on $X^u$ and $(X^u,(u\tau)^u)$ is topologically complete.
\end{corollary}
\begin{example}Recall by Theorem 6.4 of [\ref{705}] that if $X$ is a vector lattice and $A$ an ideal of $X^{\sim}$ then the absolute weak topology $\lvert\sigma\rvert(A,X)$ is a Hausdorff Lebesgue topology on $A$. This means that the topology $u\lvert \sigma\rvert(A,X)$ is the unique unbounded Hausdorff Lebesgue topology on $A$. In particular, if $X$ is a Banach lattice then $u\lvert \sigma\rvert(X^*,X)$ is the unique unbounded Hausdorff Lebesgue topology on $X^*$ so, if $X^*$ is (norm) order continuous, then $un=uaw=u\lvert\sigma\rvert(X^*,X)$ on $X^*$.
\end{example}
We next characterize the $uo$-Lebesgue property:
\begin{theorem}\label{25}
Let $(X,\tau)$ be a Hausdorff locally solid vector lattice. TFAE:
\begin{enumerate}
\item $\tau$ is $uo$-Lebesgue, i.e., $x_{\alpha}\xrightarrow{uo}0 \Rightarrow x_{\alpha}\xrightarrow{\tau}0$;
\item $\tau$ is Lebesgue and unbounded.
\end{enumerate}
\end{theorem}
\begin{proof}
Recall that the Lebesgue property is independent of the definition of order convergence. 
\\

(ii)$\Rightarrow$(i) is known since if $\tau$ is Lebesgue then $u\tau$ is $uo$-Lebesgue and, therefore, since $\tau=u\tau$, $\tau$ is $uo$-Lebesgue.
\\

Assume now that $\tau$ is $uo$-Lebesgue. Note first that this trivially implies that $\tau$ is Lebesgue. Assume now that $(x_n)$ is a disjoint sequence in $X_+$. By known results, $(x_n)$ is $uo$-null. Since $\tau$ is $uo$-Lebesgue, $(x_n)$ is $\tau$-null. Therefore, $\tau$ satisfies condition (v) of Theorem~\ref{24.1}. It therefore also satisfies (ix) which means $\tau$ is unbounded.
\end{proof}
\begin{question}
Does the above theorem remain valid if $uo$ is replaced by $uo_1$? In other words, is the $uo$-Lebesgue property independent of the definition of order convergence?
\end{question} 

There are two natural ways to incorporate unboundedness into the literature on locally solid vector lattices. The first is to take some property relating order convergence to topology and then make the additional assumption that the topology is unbounded. The other is to take said property and replace order convergence with $uo$-convergence. For the Lebesgue property, these approaches are equivalent: $\tau$ is unbounded and Lebesgue iff it is $uo$-Lebesgue (with the overlying assumption $\tau$ is Hausdorff). Later on we will study the Fatou property and see that these approaches differ.
\\

We can now strengthen Proposition~\ref{4}. Compare this with Theorem 4.20 and 4.22 in [\ref{705}]. The latter theorem says that all Hausdorff Lebesgue topologies induce the same topology on order bounded subsets. It will now be shown that, furthermore, all Hausdorff Lebesgue topologies have the same topologically closed sublattices. 
\begin{theorem}\label{26}
Let $\tau$ and $\sigma$ be Hausdorff Lebesgue topologies on a vector lattice $X$ and let $Y$ be a sublattice of $X$. Then $Y$ is $\tau$-closed in $X$ if and only if it $\sigma$-closed in $X$.
\end{theorem}
\begin{proof}
By Proposition~\ref{4}, $Y$ is $\tau$-closed in $X$ if and only if it is $u\tau$-closed in $X$ and it is $\sigma$-closed in $X$ if and only if it is $u\sigma$-closed in $X$. By Theorem~\ref{24}, $u\tau=u\sigma$, so $Y$ is $\tau$-closed in $X$ if and only if it $\sigma$-closed in $X$.
\end{proof}

Next we present a partial answer to the question of whether unbounding and passing to the order completion is the same as passing to the order completion and then unbounding. First, a proposition:
\begin{proposition}\label{22222}
Let $(X,\tau)$ be a Hausdorff locally solid vector lattice with the $uo$-Lebesgue property. Then $Y$ is a regular sublattice of $X$ iff $\tau|_Y$ is a Hausdorff $uo$-Lebesgue topology on $Y$.
\end{proposition}
\begin{proof}
The reader should convince themselves that the subspace topology defines a Hausdorff locally solid topology on $Y$; we will prove that $\tau|_Y$ is $uo$-Lebesgue when $Y$ is regular. Suppose $(y_{\alpha})$ is a net in $Y$ and $y_{\alpha}\xrightarrow{uo}0$ in $Y$. Since $Y$ is regular, $y_{\alpha}\xrightarrow{uo}0$ in $X$, so that $y_{\alpha}\xrightarrow{\tau}0$ as $\tau$ is $uo$-Lebesgue. This is equivalent to $y_{\alpha}\xrightarrow{\tau|_Y}0$.

For the converse, assume $\tau|_Y$ is Hausdorff and $uo$-Lebesgue, and $y_{\alpha}\downarrow 0$ in $Y$. Then $y_{\alpha}\xrightarrow{\tau|_Y}0$, hence $y_{\alpha}\xrightarrow{\tau}0$. Since $(y_{\alpha})$ is decreasing in $X$, $y_{\alpha}\downarrow 0$ in $X$ be [\ref{705}] Theorem 2.21. This proves that $Y$ is regular in $X$.
\end{proof}

Let $\sigma$ be a Hausdorff Lebesgue topology on a vector lattice $X$. By Theorem $4.12$ in [\ref{705}] there is a unique Hausdorff Lebesgue topology $\sigma^{\delta}$ on $X^{\delta}$ that extends $\sigma$. We have the following:
\begin{lemma}
For any Hausdorff Lebesgue topology $\tau$ on a vector lattice $X$, $u(\tau^{\delta})=(u\tau)^{\delta}=(u\tau)^u|_{X^\delta}$.
\end{lemma}
\begin{proof}
As stated, $\tau$ extends uniquely to a Hausdorff Lebesgue topology $\tau^{\delta}$ on $X^{\delta}$. $u(\tau^{\delta})$ is thus a Hausdorff $uo$-Lebesgue topology on $X^{\delta}$.
\\

Alternatively, since $\tau$ is a Hausdorff Lebesgue topology, $u\tau$ is a Hausdorff $uo$-Lebesgue topology. This topology extends uniquely to a Hausdorff Lebesgue topology $(u\tau)^{\delta}$ on $X^{\delta}$. It suffices to prove that $(u\tau)^{\delta}$ is $uo$-Lebesgue since then, by uniqueness of such topologies, it must equal $u(\tau^{\delta})$. 
\\

Since $u\tau$ is Hausdorff and $uo$-Lebesgue, it also extends to a Hausdorff $uo$-Lebesgue topology $(u\tau)^u$ on $X^u$. By page 187 of [\ref{705}], the universal completions of $X$ and $X^{\delta}$ coincide so we can restrict $(u\tau)^u$ to $X^{\delta}$. By Proposition~\ref{22222}, this gives a Hausdorff $uo$-Lebesgue topology, $(u\tau)^u|_{X^\delta}$, on $X^{\delta}$ that extends $u\tau$. By uniqueness of Hausdorff Lebesgue extensions to $X^{\delta}$, $(u\tau)^u|_{X^\delta}=(u\tau)^{\delta}$ and so $(u\tau)^{\delta}$ is $uo$-Lebesgue.
\end{proof}
In particular, if $\tau$ is also unbounded, so that $\tau$ is a Hausdorff $uo$-Lebesgue topology, then $u(\tau^{\delta})=\tau^{\delta}$. This means we can also include $uo$-Lebesgue topologies in [\ref{705}] Theorem 4.12:
\begin{corollary}
Let $\tau$ be a Hausdorff Lebesgue topology on a vector lattice $X$. Then $\tau$ is $uo$-Lebesgue iff $\tau^{\delta}$ is $uo$-Lebesgue.
\end{corollary}

\section{Minimal topologies}
In this section we will see that $uo$-convergence ``knows" exactly which topologies are minimal. Much work has been done on minimal topologies and, unfortunately, the section in [\ref{705}] is out-of-date both in terminology and sharpness of results. First we fix our definitions; they are inconsistent with [\ref{705}].
\begin{defn}
A Hausdorff locally solid topology $\tau$ on a vector lattice $X$ is said to be \textbf{\textit{minimal}} if it follows from $\tau_1\subseteq \tau$ and $\tau_1$ a Hausdorff locally solid topology that $\tau_1=\tau$.
\end{defn}
\begin{defn}
A Hausdorff locally solid topology $\tau$ on a vector lattice $X$ is said to be \textbf{\textit{least}} or, to be consistent with $[\ref{714}]$, \textbf{\textit{smallest}}, if $\tau$ is coarser than any other Hausdorff locally solid topology $\sigma$ on $X$, i.e., $\tau\subseteq \sigma$.
\end{defn}
A crucial result, not present in [\ref{705}], is Proposition 6.1 of [\ref{714}]:
\begin{proposition}\label{2000}
A minimal topology is a Lebesgue topology.
\end{proposition}
This allows us to prove the following result; the equivalence of (i) and (ii) has already been established by Theorem~\ref{25}, but is collected here for convenience.
\begin{theorem}\label{13375}
Let $\tau$ be a Hausdorff locally solid topology on a vector lattice $X$. TFAE:
\begin{enumerate}
\item $\tau$ is $uo$-Lebesgue;
\item $\tau$ is Lebesgue and unbounded;
\item $\tau$ is minimal.
\end{enumerate}
\end{theorem}
\begin{proof}

Suppose $\tau$ is minimal. By the last proposition, it is Lebesgue. It is also unbounded since $u\tau$ is a Hausdorff locally solid topology and $u\tau \subseteq \tau$. Minimality forces $\tau=u\tau$.
\\

Conversely, suppose that $\tau$ is Hausdorff, Lebesgue and unbounded and that $\sigma \subseteq \tau$ is a Hausdorff locally solid topology. It is clear that $\sigma$ is then Lebesgue and hence $\sigma$-Fatou. By Theorem~\ref{24.1}(iii), $\tau$ is coarser than $\sigma$. Therefore, $\tau=\sigma$ and $\tau$ is minimal.
\end{proof}
In conjunction with Theorem~\ref{24}, we deduce the (already known) fact that minimal topologies, if they exist, are unique. They exist if and only if $X$ admits a Hausdorff Lebesgue topology.

 We also get many generalizations of results on unbounded topologies in Banach lattices. The following is part of [\ref{705}] Theorem 7.65:
\begin{theorem}
If $X$ is an order continuous Banach lattice then $X$ has a least topology.
\end{theorem}
The least topology on an order continuous Banach lattice is, simply, $un$. This proves that $un$ is ``special", and also that it has been implicitly studied before. If $X$ is any Banach lattice then, since the topology $u\lvert\sigma\rvert(X^*,X)$ is Hausdorff and $uo$-Lebesgue, it also has a minimality property:
\begin{lemma}
Let $X$ be a Banach lattice. $u\lvert\sigma\rvert(X^*,X)$ is the (unique) minimal topology on $X^*$.
\end{lemma}


\section{Local convexity and dual spaces of $uo$-Lebesgue topologies}

We now make some remarks about $u\tau$-continuous functionals and, surprisingly, generalize many results in [\ref{706}] whose presented proofs rely heavily on AL-representation theory and the norm.
\\

First recall that by [\ref{705}] Theorem 2.22, if $\sigma$ is a locally solid topology on $X$ then $(X,\sigma)^* \subseteq X^{\sim}$ as an ideal. $(X,\sigma)^*$ is, therefore, an order complete vector lattice in its own right. Here $(X,\sigma)^*$ stands for the topological dual and $X^{\sim}$ for the order dual.
\begin{proposition}\label{14}
$(X,u\tau)^*\subseteq (X,\tau)^*$ as an ideal.
\end{proposition}
\begin{proof}

It is easy to see that the set of all $u\tau$-continuous functionals in $(X,\tau)^*$ is a linear subspace. Suppose that $\varphi$ in $(X,\tau)^*$ is $u\tau$-continuous; we will show that $\lvert\varphi\rvert$ is also $u\tau$-continuous. Fix $\varepsilon>0$ and let $\{U_i\}_{i\in I}$ be a solid base for $\tau$ at zero. By $u\tau$-continuity of $\varphi$, one can find an $i\in I$ and $u>0$ such that $\lvert\varphi(x)\rvert< \varepsilon$ whenever $x \in U_{i,u}$. Fix $x \in U_{i,u}$. Since $U_{i,u}$ is solid, $\lvert y\rvert \leq \lvert x\rvert$ implies $y \in U_{i,u}$ and, therefore, $\lvert \varphi(y)\rvert < \varepsilon$. By the Riesz-Kantorovich formula, we get that
\begin{equation}
\bigl\lvert\lvert\varphi\rvert(x)\bigl\rvert \leq \lvert\varphi\rvert\bigl(\lvert x\rvert\bigl)=\text{sup}\Bigl\{\bigl\lvert\varphi(y)\bigl\rvert\ : \lvert y\rvert \leq \lvert x\rvert\Bigl\} \leq \varepsilon.
\end{equation}
It follows that $\lvert \varphi\rvert$ is $u\tau$-continuous and, therefore, the set of all $u\tau$-continuous functionals in $(X,\tau)^*$ forms a sublattice. It is straightforward to see that if $\varphi\in (X,\tau)^*_+$ is $u\tau$-continuous and $0\leq \psi \leq \varphi$ then $\psi$ is also $u\tau$-continuous and, thus, the set of all $u\tau$-continuous functionals in $(X,\tau)^*$ is an ideal.
\end{proof}

We next need some definitions. Our definition of discrete element is slightly different than [\ref{714}] since we require them to be positive and non-zero. It is consistent with [\ref{705}]. 
\begin{defn}
Let $X$ be a vector lattice. $x>0$ in $X$ is called a \textbf{\textit{discrete element}} or \textbf{\textit{atom}} if the ideal generated by $x$ equals the linear span of $x$.
\end{defn}
\begin{defn}
A vector lattice $X$ is \textbf{\textit{discrete}} or \textbf{\textit{atomic}} if there is a complete disjoint system $\{x_i\}$ consisting of discrete elements in $X_+$, i.e., $x_i\wedge x_j=0$ if $i\neq j$ and $x\in X$, $x\wedge x_i=0$ for all $i$ implies $x=0$.
\end{defn}

By [\ref{705}] Theorem 1.78, $X$ is atomic if and only if $X$ is lattice isomorphic to an order dense sublattice of some vector lattice of the form $\mathbb{R}^A$.
\\

The next result is an effortless generalization of Theorem 5.2 in [\ref{706}]. In the upcoming results we consider the $0$-vector lattice to be atomic.
\begin{lemma}\label{1289}
Let $\tau$ be a Hausdorff $uo$-Lebesgue topology on a vector lattice $X$. $\tau$ is locally convex if and only if $X$ is atomic. Moreover, if $X$ is atomic then a Hausdorff $uo$-Lebesgue topology exists, it is least, and it is the topology of pointwise convergence.
\end{lemma}
\begin{proof}
By page 291 of [\ref{714}], a pre-$L_0$ space is the same as a vector lattice that admits a Hausdorff Lebesgue topology; the first part of our lemma is just a re-wording of Proposition 3.5 in [\ref{714}]. The moreover part follows from Theorem 7.70 in [\ref{705}] (remember, ``minimal" in [\ref{705}] means ``least"). Actually, knowing that Hausdorff $uo$-Lebesgue topologies are Lebesgue and disjoint sequences are null, the entire lemma can be deduced from the statement and proof of Theorem 7.70.
\end{proof}
Consider Remark 4.15 in [\ref{706}]. It is noted that $\ell_{\infty}$ is atomic yet $un$-convergence is not the same as pointwise convergence. Lemma~\ref{1289} tells us that $\ell_{\infty}$ does admit a least topology that coincides with the pointwise convergence. Since $\ell_{\infty}$ is a dual Banach lattice, $u\lvert\sigma\rvert(\ell_{\infty},\ell_1)$ is defined and must be the least topology on $\ell_{\infty}$. This is an example where $X^*$ is not an order continuous Banach lattice but $u\lvert \sigma\rvert(X^*,X)$ is still a least topology.

\begin{theorem}\label{atomic}
$uo$-convergence in a vector lattice $X$ agrees with the convergence of a locally convex-solid topology on $X$ iff $X$ is atomic.
\end{theorem}
\begin{proof}
Suppose $uo$-convergence agrees with the convergence of a locally convex-solid topology $\tau$. Since $uo$-limits are unique, $\tau$ is Hausdorff. Clearly, $\tau$ is $uo$-Lebesgue, so, by Lemma~\ref{1289}, $X$ is atomic.
\\

Suppose $X$ is atomic. By [\ref{705}] Theorem 1.78, $X$ is lattice isomorphic to an order dense sublattice of a vector lattice of the form $\mathbb{R}^A$. Since $uo$-convergence is preserved through the onto isomorphism and the order dense embedding into $\mathbb{R}^A$, we may assume that $X\subseteq \mathbb{R}^A$. It was noted in [\ref{dual spaces}] that $uo$-convergence in $\mathbb{R}^A$ is just pointwise convergence. Using Theorem 3.2 of [\ref{708}], it is easy to see that the restriction of pointwise convergence to $X$ agrees with $uo$-convergence on $X$. Hence, $uo$-convergence agrees with the convergence of a locally convex-solid topology. See [\ref{1337}] for an alternative proof that $uo$-convergence in atomic vector lattices is topological.
\end{proof}
The following result is known, but nevertheless follows immediately:
\begin{corollary}
A vector lattice $X$ is atomic iff it is lattice isomorphic to a regular sublattice of some vector lattice of the form $\mathbb{R}^A$.
\end{corollary}
\begin{proof}
The forward direction follows from Theorem 1.78 of [\ref{705}]. For the converse, combine Theorem~\ref{atomic} with Theorem 3.2 of [\ref{708}].
\end{proof}
Theorem 7.71 in [\ref{705}] is a perfectly reasonable generalization of [\ref{706}] Corollary 5.4(ii) since the $un$-topology in order continuous Banach lattices is least. What we want, however, is to replace the least topology assumption in Theorem 7.71 with the assumption that the topology is minimal. The reason being that $uo$-convergence can ``detect" if a topology is minimal, but not necessarily if it is least. To prove Corollary 5.4 in [\ref{706}] the authors go through the theory of dense band decompositions. A similar theory of $\tau$-dense band decompositions can be developed, but there is an easier proof of this result utilizing the recent paper [\ref{707}].
\begin{proposition}
Let $\tau$ be a $uo$-Lebesgue topology on a vector lattice $X$. If $0\neq \varphi\in (X,\tau)^*$ then $\varphi$ is a linear combination of the coordinate functionals of finitely many atoms.
\end{proposition}
\begin{proof}
Suppose $0\neq \varphi\in (X,\tau)^*$. Since $\tau$ is $uo$-Lebesgue and $\varphi$ is $\tau$-continuous, $\varphi(x_{\alpha})\rightarrow 0$ whenever $x_{\alpha}\xrightarrow{uo}0$. The conclusion now follows from  Proposition 2.2 in [\ref{707}].
\end{proof}


\section{Fatou topologies}

Using the canonical base described in Theorem~\ref{1}, it is trivial to verify that if $\tau$ has the Fatou property then so does $u\tau$. In analogy with the Fatou property, it is natural to consider topologies that have a base at zero consisting of solid $uo$-closed sets. Surprisingly, this does not lead to a new concept:

\begin{lemma}\label{8882}
Let $A\subseteq X$ be a solid subset of a vector lattice $X$. $A$ is (sequentially) $o$-closed if and only if it is (sequentially) $uo$-closed.
\end{lemma}
\begin{proof}
If $A$ is $uo$-closed then it is clearly $o$-closed. Suppose $A$ is $o$-closed, $(x_{\alpha})\subseteq A$ and $x_{\alpha}\xrightarrow{uo}x$. We must prove $x \in A$. By continuity of lattice operations, $|x_{\alpha}|\wedge|x|\xrightarrow{uo}|x|$, so that $|x_{\alpha}|\wedge|x|\xrightarrow{o}|x|$. Since $A$ is solid, $(|x_{\alpha}|\wedge|x|)\subseteq A$, and since $A$ is $o$-closed we conclude that $|x|\in A$. Finally, using the solidity of $A$ again, we conclude that $x\in A$. Sequential arguments are analogous.

\end{proof}

A simliar proof to Lemma~\ref{8882} gives the following. Compare with [\ref{704}] Lemma 2.8.
\begin{lemma}\label{546}
If $x_{\alpha}\xrightarrow{u\tau}x$ then $\lvert x_{\alpha}\rvert \wedge\lvert x\rvert \xrightarrow{\tau}\lvert x\rvert$. In particular, $\tau$ and $u\tau$ have the same (sequentially) closed solid sets.
\end{lemma}

This leads to the following elegant result:
\begin{theorem}\label{solid}
Let $\tau$ and $\sigma$ be Hausdorff Lebesgue topologies on a vector lattice $X$ and let $A$ be a solid subset of $X$. Then $A$ is (sequentially) $\tau$-closed if and only if it is (sequentially) $\sigma$-closed.
\end{theorem}
\begin{proof}
Suppose $A$ is $\tau$-closed. By Lemma~\ref{546}, $A$ is $u\tau$-closed. Since $X$ can admit only one unbounded Hausdorff Lebesgue topology, $u\sigma=u\tau$ and, therefore, $A$ is $u\sigma$-closed. Since $u\sigma \subseteq \sigma$, $A$ is $\sigma$-closed. Sequential arguments are analogous.
\end{proof}
\begin{remark}
It is well known that locally convex topologies consistent with a given dual pair have the same closed convex sets. Theorem~\ref{solid} is a similar result for locally solid topologies. It also motivates Question~\ref{Qhi}.
\end{remark}
We can also strengthen Lemma 3.6 in [\ref{716}]. For properties and terminology involving Riesz seminorms, the reader is referred to [\ref{705}].
\begin{lemma}
Let $X$ be a vector lattice and suppose $\rho$ is a Riesz seminorm on $X$ satisfying the Fatou property. Then $x_{\alpha}\xrightarrow{uo}x\Rightarrow \rho(x)\leq \liminf\rho(x_{\alpha})$.
\end{lemma}
\begin{proof}
First we prove the statement for order convergence. Assume $x_{\alpha}\xrightarrow{o}x$ and pick a dominating net $y_{\beta}\downarrow 0$. Fix $\beta$ and find $\alpha_0$ such that $\lvert x_{\alpha}-x \rvert\leq y_{\beta}$ for all $\alpha\geq \alpha_0$. Since \begin{displaymath}
(\lvert x\rvert-y_{\beta})^+\leq \lvert x_{\alpha}\rvert,
\end{displaymath}
we conclude that $\rho((\lvert x\rvert-y_{\beta})^+) \leq \rho(x_{\alpha})$. Since this holds for all $\alpha\geq \alpha_0$ we can conclude  that $\rho((\lvert x\rvert-y_{\beta})^+) \leq \liminf\rho(x_{\alpha})$. Since $\rho$ is Fatou and $0 \leq (\lvert x\rvert-y_{\beta})^+\uparrow \lvert x\rvert$ we conclude that $\rho((\lvert x\rvert-y_{\beta})^+)\uparrow \rho(x)$ and so $\rho(x)\leq \liminf \rho(x_{\alpha})$. 
\\

Now assume that $x_{\alpha}\xrightarrow{uo}x$. Then $\lvert x_{\alpha}\rvert \wedge \lvert x\rvert \xrightarrow{o}\lvert x\rvert.$ Using the above result and properties of Riesz seminorms, $\rho(x)=\rho(\lvert x\rvert )\leq \liminf \rho(\lvert x_{\alpha}\rvert \wedge\lvert x\rvert)\leq \liminf\rho(\lvert x_{\alpha}\rvert)=\liminf\rho(x_{\alpha})$.
\end{proof}



We next investigate how unbounded Fatou topologies lift to the order completion. Theorem 4.12 of [\ref{705}] asserts that if $\sigma$ is a Fatou topology on a vector lattice $X$ then $\sigma$ extends uniquely to a Fatou topology $\sigma^{\delta}$ on $X^{\delta}$. We will use this notation in the following theorem. 
\begin{proposition}
Let $X$ be a vector lattice and $\tau$ a Fatou topology on $X$. Then $u(\tau^{\delta})=(u\tau)^{\delta}$.
\end{proposition}
\begin{proof}
Since $\tau$ is Fatou, $\tau$ extends uniquely to a Fatou topology $\tau^{\delta}$ on $X^{\delta}$. Clearly, $u(\tau^{\delta})$ is still Fatou. Suppose $(x_{\alpha})$ is a net in $X$ and $x \in X$. By Corollary~\ref{18.1}, $x_{\alpha}\xrightarrow{u(\tau^{\delta})}x$ in $X^{\delta}$ if and only if $x_{\alpha}\xrightarrow{u\tau}x$ in $X$.
\\

Since $\tau$ is Fatou, so is $u\tau$. Therefore, $u\tau$ extends uniquely to a Fatou topology $(u\tau)^{\delta}$ on $X^{\delta}$. Suppose $(x_{\alpha})$ is a net in $X$ and $x \in X$. Then $x_{\alpha}\xrightarrow{(u\tau)^{\delta}}x$ is the same as $x_{\alpha}\xrightarrow{u\tau}x$.
\\

Thus, $(u\tau)^{\delta}$ and $u(\tau^{\delta})$ are two Fatou topologies on $X^{\delta}$ that agree with the Fatou topology $u\tau$ when restricted to $X$. By uniqueness of extension $u(\tau^{\delta})=(u\tau)^{\delta}$.
\end{proof}
\begin{defn}
A locally solid vector lattice $(X,\tau)$ is said to be \textbf{\textit{weakly Fatou}} if $\tau$ has a base $\{U_i\}$ at zero consisting of solid sets with the property that for all $i$ there exists $k_i \geq 1$ such that whenever $(x_{\alpha})$ is a net in $U_i$ and $x_{\alpha}\xrightarrow{o}x$ we have $x\in k_iU_i$.
\end{defn}
\begin{remark}
It is easily seen that for solid $U$ and $k\geq 1$, the property that $x\in kU$ whenever $(x_{\alpha})$ is a net in $U$ and $x_{\alpha}\xrightarrow{o}x$ is equivalent to the property that $x\in kU$ whenever $(x_{\alpha})$ is a net in $U$ and $0\leq x_{\alpha}\uparrow x$. Also, note that a Banach lattice $X$ is weakly Fatou if and only if there exists $k \geq 1$ such that $\|x\|\leq k\sup_{\alpha}\|x_{\alpha}\|$ whenever $0 \leq x_{\alpha}\uparrow x$ in $X$.
\end{remark}

Clearly, Fatou topologies are weakly Fatou. The next theorem, and one direction of its proof, is motivated by [\ref{707}] Proposition 3.1.
\begin{theorem}\label{86868686}
Suppose $(X,\tau)$ is Hausdorff and weakly Fatou. Then $\tau$ is Levi iff $(X,\tau)$ is boundedly $uo$-complete.
\end{theorem}
\begin{proof}
If $\tau$ is Levi, $X$ is order complete by page 112 of [\ref{705}].
\\

Let $(x_{\alpha})$ be a $\tau$-bounded $uo$-Cauchy net in $X$.  By considering the positive and negative parts, respectively, we may assume that $x_{\alpha} \geq 0$ for each $\alpha$. For each $y \in X_+$, since $\lvert x_{\alpha}\wedge y-x_{\alpha'}\wedge y\rvert \leq \lvert x_{\alpha}-x_{\alpha'}\rvert \wedge y$, the net $(x_{\alpha}\wedge y)$ is order Cauchy and hence order converges to some $u_y\in X_+$. The net $(u_y)_{y\in X_+}$ is directed upwards; we show it is $\tau$-bounded.
Let $U$ be a solid $\tau$-neighbourhood of zero with the property that there exists $k\geq 1$ with $x\in kU$ whenever $(x_{\alpha})$ is a net in $U$ and $x_{\alpha}\xrightarrow{o}x$. Since $(x_{\alpha})$ is $\tau$-bounded, there exists $\lambda>0$ such that $(x_{\alpha})\subseteq \lambda U$. Since $0\leq x_{\alpha}\wedge y \leq x_{\alpha} \in \lambda U$, $x_{\alpha}\wedge y \in \lambda U$ for all $y$ and $\alpha$ by solidity. We conclude that $u_y \in \lambda kU$ for all $y$, so that $(u_y)$ is $\tau$-bounded.
\\

Since $\tau$ is Levi, $(u_y)$ increases to an element $u \in X$. Fix $y \in X_+$. For any $\alpha,\alpha'$, define
\begin{equation}
x_{\alpha,\alpha'}=\sup_{\beta \geq \alpha, \beta'\geq \alpha'}\lvert x_{\beta}-x_{\beta'}\rvert\wedge y.
\end{equation}
Since $(x_{\alpha})$ is $uo$-Cauchy, $x_{\alpha,\alpha'}\downarrow 0$. Also, for any $z\in X_+$ and any $\beta \geq \alpha, \beta' \geq \alpha'$,
\begin{equation}
\lvert x_{\beta}\wedge z-x_{\beta'}\wedge z\rvert\wedge y \leq x_{\alpha,\alpha'}.
\end{equation}
Taking order limit first in $\beta'$ and then over $z \in X_+$, we obtain $\lvert x_{\beta}-u\rvert \wedge y \leq x_{\alpha,\alpha'}$ for any $\beta \geq \alpha$. This implies that $(x_{\alpha})$ $uo$-converges to $u$.
\\

For the converse, assume $(X,\tau)$ is boundedly $uo$-complete and let $(x_{\alpha})$ be a positive increasing $\tau$-bounded net in $X$. Following the proof of [\ref{705}] Theorem 7.50, it is easily seen that $(x_{\alpha})$ is dominable. By [\ref{705}] Theorem 7.37, $(x_{\alpha})$ has supremum in $X^u$, hence is $uo$-Cauchy in $X^u$, hence is $uo$-Cauchy in $X$. Since $(x_{\alpha})$ is $\tau$-bounded, $x_{\alpha}\xrightarrow{uo}x$ in $X$ for some $x\in X$. Since $(x_{\alpha})$ is increasing, $x=\sup x_{\alpha}$. This proves that $\tau$ is Levi.
\end{proof}

\section{Unbounded convergence witnessed by ideals}
In this section we see which results in [\ref{710}] move to the general setting.
\\

To decide whether $x_{\alpha} \xrightarrow{u\tau}x$ one has to check if $\lvert x_{\alpha}-x\rvert\wedge u\xrightarrow{\tau}0$ for every ``test" vector $u \in X_+$. A natural question is, why do we take our test vectors from $X_+$? In this section we study unbounded convergence against a smaller test set.
\begin{defn}\label{33}
Let $(X,\tau)$ be a locally solid vector lattice and $A \subseteq X$ an ideal. We say a net $(x_{\alpha})$ \textbf{\textit{unbounded $\tau$-converges to $x$ with respect to $A$}} if $\lvert x_{\alpha}-x\rvert \wedge \lvert a\rvert \xrightarrow{\tau}0$ for all $a \in A$ or, equivalently, if $\lvert x_{\alpha}-x\rvert \wedge  a \xrightarrow{\tau}0$ for all $a \in A_+$.
\end{defn}
\begin{remark} The assumption that $A$ is an ideal in the last definition presents no loss in generality since $\lvert x_{\alpha}-x\rvert \wedge \lvert a\rvert \xrightarrow{\tau}0$ for all $a \in A$ if and only if $\lvert x_{\alpha}-x\rvert \wedge \lvert a\rvert \xrightarrow{\tau}0$ for all $a \in I(A)$
\end{remark}

\begin{proposition}\label{34}
If $A$ is an ideal of a locally solid vector lattice $(X,\tau)$ then the unbounded $\tau$-convergence with respect to $A$ is a topological convergence on $X$. Moreover, the corresponding topology, $u_A\tau$, is locally solid.
\end{proposition}
\begin{proof}
A minor modification of the proof of Theorem~\ref{1}. The base neighbourhoods are absorbing since $\tau$ is defined on $X$.
\end{proof}
Notice the change in notation from [\ref{710}]. One may think of $u$ and $u_A$ as maps from the set of locally solid topologies on $X$ to itself. In particular, $u_Au\tau$ should make sense and equal $u_A(u(\tau))$. This is why this notation is chosen. It is evident that $u=u_X$ so this subject is more general than the previous sections of the paper.
\\

It is clear that $(u_A\tau)|_A=u(\tau|_A)$. It can be checked that if $A$ and $B$ are ideals of a locally solid vector lattice $(X,\tau)$ then $u_A(u_B\tau)=u_B(u_A\tau)=u_{A\cap B}\tau$. In particular, $u_A\tau$ is always unbounded. Notice also that if $A\subseteq B$ then $u_A\tau \subseteq u_B\tau$. 
\\

Since $u_A\tau$ is locally solid, Proposition 1.2 in [\ref{710}] comes for free. We now present the analog of Proposition 1.4 in [\ref{710}]:
\begin{proposition}\label{35}
Let $A$ be an ideal of a locally solid vector lattice $(X,\tau)$. Then $u_A\tau$ is Hausdorff iff $\tau$ is Hausdorff and $A$ is order dense in $X$. 
\end{proposition}
\begin{proof}
Routine modification of the proof of Proposition 1.4 in [\ref{710}].
\end{proof}
We now move on to the analog of [\ref{710}] Proposition 2.2.
\begin{proposition}\label{36}
Suppose $A$ and $B$ are ideals of a locally solid vector lattice $(X,\tau)$. If $\overline{A}^{\tau}=\overline{B}^{\tau}$ then the topologies $u_A\tau$ and $u_B\tau$ on $X$ agree.
\end{proposition}
\begin{proof}
It suffices to show that $u_A\tau=u_{\overline{A}}\tau$, where, for notational simplicity, $\overline{A}$ denotes the $\tau$-closure of $A$ in $X$. Let $(x_{\alpha})$ be a net in $X$. Clearly, if  $x_{\alpha} \xrightarrow{u_{\overline{A}}\tau}0$ then $x_{\alpha} \xrightarrow{u_A\tau}0$. To prove the converse, suppose that $x_{\alpha} \xrightarrow{u_A\tau}0$. Fix $y \in \overline{A}^{\tau}_+$, a solid base neighbourhood $V$ of zero for $\tau$, and a solid base neighbourhood $U$ of zero for $\tau$ with $U+U\subseteq V$. By definition, there exists $a \in A$ such that $a \in y+U$. WLOG $a \in A_+$ because, by solidity, $\lvert\lvert a\rvert-y\rvert \leq \lvert a-y\rvert\in U$ implies $\lvert a\rvert \in y+U$. By assumption, $|x_{\alpha}| \wedge a \xrightarrow{\tau}0$. This implies that there exists $\alpha_0$ such that $|x_{\alpha}| \wedge a \in U$ whenever $\alpha\geq \alpha_0$. It follows by solidity that\begin{equation}
|x_{\alpha}| \wedge y=|x_{\alpha}| \wedge (y-a+a) \leq |x_{\alpha}| \wedge \lvert y-a\rvert+|x_{\alpha}| \wedge a \in U+U \subseteq V,
\end{equation}
so that $x_{\alpha} \xrightarrow{u_{\overline{A}}\tau}0$.
\end{proof}

\begin{theorem}\label{555}
Let $X$ be a vector lattice and $Y_1,Y_2 \subseteq X$ order dense ideals of $X$. Suppose $\tau_1$ and $\tau_2$ are Hausdorff Lebesgue topology on $X$. Then the topologies $u_{Y_1}\tau_1$ and $u_{Y_2}\tau_2$ agree on $X$. Moreover, this topology is the minimal topology on $X$ so is Hausdorff and $uo$-Lebesgue.
\end{theorem}
\begin{proof}
$u_{Y_1}\tau_1$ is a Hausdorff locally solid topology on $X$ that is coarser than $u\tau_1$. Since $u\tau_1$ is minimal, this forces $u_{Y_1}\tau_1=u\tau_1$. By uniqueness of minimal topologies, $u\tau_1=u\tau_2$, and, by similar arguments, $u_{Y_2}\tau_2=u\tau_2$.
\end{proof}
We next generalize Corollary 4.6 of [\ref{706}].

\begin{lemma}\label{38.8}
Suppose $Y$ is a sublattice of a vector lattice $X$. If $\tau$ is a Hausdorff Lebesgue topology on $X$ then $u(\tau|_Y)=(u\tau)|_Y$.
\end{lemma}
\begin{proof}
It is clear that $u(\tau|_Y)\subseteq (u\tau)|_Y$.
\\

Suppose $(y_{\alpha})$ is a net in $Y$ and $y_{\alpha}\xrightarrow{u(\tau|_Y)}0$. Since $Y$ is majorizing in $I(Y)$, the ideal generated by $Y$ in $X$, $y_{\alpha}\xrightarrow{u_{I(Y)}\tau}0$. By Theorem 1.36 of [\ref{709}], $I(Y)\oplus I(Y)^d$ is an order dense ideal in $X$. Let $v \in (I(Y)\oplus I(Y)^d)_+$. Then $v=a+b$ where $a \in I(Y)$ and $b\in I(Y)^d$. Notice $\lvert y_{\alpha}\rvert \wedge v\leq \lvert y_{\alpha}\rvert \wedge \lvert a\rvert+\lvert y_{\alpha}\rvert \wedge \lvert b\rvert=\lvert y_{\alpha}\rvert \wedge \lvert a\rvert \xrightarrow{\tau}0$. This proves that $y_{\alpha}\xrightarrow{u_{I(Y)\oplus I(Y)^d}\tau}0$. We conclude that $(u_{I(Y)\oplus I(Y)^d}\tau)|_Y \subseteq u(\tau|_Y)$. Since the other inclusion is obvious, $(u_{I(Y)\oplus I(Y)^d}\tau)|_Y = u(\tau|_Y)$.
\\

Since $I(Y)\oplus I(Y)^d$ is order dense in $X$, $u_{I(Y)\oplus I(Y)^d}\tau$ is a  Hausdorff locally solid topology on $X$. Clearly, $u_{I(Y)\oplus I(Y)^d}\tau \subseteq u\tau$ so, since $u\tau$ is a Hausdorff $uo$-Lebesgue topology and hence minimal, $u_{I(Y)\oplus I(Y)^d}\tau = u\tau$. This proves the claim.
\end{proof}
The next proposition is an analogue of [\ref{704}] Lemma 2.11.
\begin{proposition}\label{38}
Suppose $(X,\tau)$ is a locally solid vector lattice and $E\subseteq X_+$. Then $x_{\alpha}\xrightarrow{u_{\overline{I(E)}^{\tau}}\tau}x$ if and only if $\lvert x_{\alpha}-x\rvert \wedge e \xrightarrow{\tau}0$ for all $e\in E$. In particular, if there exists $e\in X_+$ such that  $\overline{I}_e^{\tau}=X$ then $x_{\alpha}\xrightarrow{u\tau}0$ iff $|x_{\alpha}|\wedge e\xrightarrow{\tau}0$.
\end{proposition}

Next we present an easy generalization of  Corollary 3.2 in [\ref{710}]. 
\begin{corollary}
Suppose $A$ is a $\tau$-closed ideal of a metrizable locally solid vector lattice $(X,\tau)$. Suppose that $e\in A_+$ is such that $\overline{I_e}^{\tau}=A$. If $x_{\alpha}\xrightarrow{u_A\tau} 0$ in $X$ then there exists $\alpha_1<\alpha_2<\dots$ such that $x_{\alpha_n}\xrightarrow{u_A\tau}0$. 
\end{corollary}
Next we present a more general version of [\ref{710}] Theorem 6.7. Using the machinery we have built, the proof is very simple.

\begin{theorem}\label{39}
Let $\tau$ be a Hausdorff Lebesgue topology on an order complete vector lattice $X$. Let $X^u$ be the universal completion of $X$ and $\sigma$ the unique Hausdorff Lebesgue topology on $X^u$. Then for every net $(x_{\alpha})$ in $X^u$, $x_{\alpha}\xrightarrow{\sigma}0$ iff $|x_{\alpha}|\wedge u\xrightarrow{\tau}0$ for all $u\in X_+$.
\end{theorem}
\begin{proof}
Since $X$ admits a Hausdorff Lebesgue topology, $X^u$ admits a unique Hausdorff Lebesgue topology by [\ref{705}] Theorems 7.53 and 7.54. Since $X$ is order complete, $X$ is an order dense ideal of $X^u$. Combine Theorem~\ref{555} with Corollary~\ref{24.2}.
\end{proof}
Note that we can replace the order completeness assumption with $\tau$-completeness because in a topologically complete Hausdorff vector lattice, the Lebesgue property implies order completeness.  
\\

Next we look for analogues of Propositions 9.1 and 9.2 in [\ref{710}]. Compare them with Theorem~\ref{23}, Corollary~\ref{10201} and Proposition~\ref{23.1}.
\begin{proposition}\label{1020}
Let $A$ be an ideal of a locally solid vector lattice $(X,\tau)$. TFAE:
\begin{enumerate}
\item $(A,\tau|_A)$ satisfies the pre-Lebesgue property;
\item Every disjoint sequence in $X$ is $u_A\tau$-null;
\item Every disjoint net in $X$ is $u_A\tau$-null; 
\item $(X,u_A\tau)$ satisfies the pre-Lebesgue property.
\end{enumerate}
\end{proposition}
\begin{proof}
To prove that (i)$\Rightarrow$(ii) let $(x_n)$ be a disjoint sequence in $X$. Then for every $a\in A_+$, $\lvert x_n\rvert \wedge a$ is an order bounded disjoint sequence in $A$ and hence $\tau$-converges to zero by Theorem~\ref{23}. This proves $x_n \xrightarrow{u_A\tau}0$. An argument already presented in the proof of Theorem~\ref{23} proves (ii)$\Leftrightarrow$(iii). (ii)$\Rightarrow$(iv) is obvious.
\\

(iv)$\Rightarrow$(i): Suppose $u_A\tau$ is a pre-Lebesgue topology on $X$. We first show that $(u_A\tau)|_A$ is a pre-Lebesgue topology on $A$. We again use Theorem~\ref{23}. Let $(a_n)$ be a disjoint order bounded sequence in $A$. Then $(a_n)$ is also a disjoint order bounded sequence in $X$ and hence $a_n \xrightarrow{u_A\tau}0$. Thus $(u_A\tau)|_A$ satisfies (iii) of Theorem~\ref{23} and we conclude that $(u_A\tau)|_A$ is pre-Lebesgue. Next notice that $(A,(u_A\tau)|_A)=(A,u(\tau|_A))$, so $u(\tau|_A)$ has the pre-Lebesgue property. Finally, apply Corollary~\ref{10201}.
\end{proof}
\begin{proposition}\label{41}
Let $A$ be an ideal of a Hausdorff locally solid vector lattice $(X,\tau)$. TFAE:
\begin{enumerate}
\item $\tau|_A$ is Lebesgue;
\item $u(\tau|_A)$ is Lebesgue;
\item $u(\tau|_A)$ is $uo$-Lebesgue;
\item $u_A\tau$ is $uo$-Lebesgue;
\item $u_A\tau$ is Lebesgue.
\end{enumerate}
\end{proposition}
\begin{proof}
The equivalence of (i), (ii) and (iii) has already been proven.
\\

(i)$\Rightarrow$(iv): Suppose $x_{\alpha}\xrightarrow{uo}0$ in $X$ where $(x_{\alpha})$ is a net in $X$. Fix $a\in A_+$. Then $\lvert x_{\alpha}\rvert\wedge a \xrightarrow{uo}0$ in $X$ and hence in $A$ since $A$ is an ideal. Since the net $(\lvert x_{\alpha}\rvert\wedge a)$ is order bounded in $A$, this is equivalent to $\lvert x_{\alpha}\rvert \wedge a \xrightarrow{o}0$ in $A$. Since $\tau|_A$ is Lebesgue this means $\lvert x_{\alpha}\rvert \wedge a \xrightarrow{\tau}0$. We conclude that $x_{\alpha}\xrightarrow{u_A\tau}0$ and, therefore, $u_A\tau$ is $uo$-Lebesgue. (iv)$\Rightarrow$(v) is trivial. (v)$\Rightarrow$(ii) since the restriction of a Lebesgue topology to a regular sublattice is Lebesgue, and $(u_A\tau)|_A=u(\tau|_A)$.
\end{proof}

\subsection*{Acknowledgements and further remarks.} The author would like to thank Dr.~Vladimir Troitsky for valuable comments and mentorship. 
\\

After the work on this paper was essentially complete, the author learned of the recent preprints [\ref{MNVL}] and [\ref{utau}]. The latter preprint also focuses on unbounded locally solid topologies, and there is a minor overlap. Specifically, it includes versions of our Theorem~\ref{1}, Lemma~\ref{18}, and Corollary~\ref{18.1}, as well as special cases of Lemma~\ref{9}, Lemma~\ref{38.8}, and Proposition~\ref{38}.
\\

Within six months of submitting this paper, I submitted the sequel papers [\ref{KT}] and [\ref{Tay2}]. Since those papers were accepted earlier than this, there are some minor misalignments of the references. In [\ref{thesis}] I have combined, extended, and reordered the results.

\end{document}